\documentclass[12pt]{article}

\usepackage{comment,caption}
\usepackage{pgf,tikz}
\usepackage{pgfplots}
\pgfplotsset{compat = 1.7}

\usetikzlibrary{arrows.meta}
\usetikzlibrary{arrows}

\usepackage{amsmath,amssymb,amsthm,cite}
\usepackage[colorlinks,citecolor=red,urlcolor=blue,linkcolor=blue]{hyperref}
\usepackage{mathrsfs,dsfont}
\usepackage{graphicx}
\usepackage{enumitem}
\usepackage[font={small},width=1\textwidth]{caption}

\newtheorem{nummer}{ }
\newtheorem{thm}[nummer]{\bf Theorem}
\newtheorem{prp}[nummer]{\bf Proposition}
\newtheorem{lem}[nummer]{\bf Lemma}
\newtheorem{cor}[nummer]{\bf Corollary}
\newtheorem{fct}[nummer]{\bf Fact}

\newcommand{\ie} {\sl i.e.}

\newcommand{\crat}{\operatorname{cr}}

\newcommand{\kringl} {\raisebox{0.3ex}{\text{\tiny{$\circ$}}}}

\newcommand{\N}{\mathds{N}}

\newcommand{\R}{\mathds{R}}

\newcommand{\nO}{\mathscr{O}}

\newcommand{\dritt}[2]{#1\mathbin{\#}#2}
\newcommand{\Inv}[1]{I_{#1}}

\newcommand{\C}{\Gamma}
\newcommand{\Cabc}{\C_{\alpha,\beta,\gamma}}
\newcommand{\Cab}{\C_{a,b}}

\newcommand{\konj}[1]{\bar{#1}}

\makeatletter
\def\opargproof[#1]{\par\noindent {\bf #1 }}
 \makeatother

\definecolor{darkgreen}{rgb}{0,.6,0}

\begin{document}
\begin{center}
{\Large\bf Constructing Cubic Curves with Involutions} 

\medskip

{\small Lorenz Halbeisen}\\[1.2ex] 
{\scriptsize Department of Mathematics, ETH Zentrum,
R\"amistrasse\;101, 8092 Z\"urich, Switzerland\\ lorenz.halbeisen@math.ethz.ch}\\[1.8ex]
{\small Norbert Hungerb\"uhler}\\[1.2ex] 
{\scriptsize Department of Mathematics, ETH Zentrum,
R\"amistrasse\;101, 8092 Z\"urich, Switzerland\\ norbert.hungerbuehler@math.ethz.ch}
\end{center}

\hspace{5ex}{\small{\it key-words\/}: cubic curve, line involution, ruler constructions, elliptic curve}

\hspace{5ex}{\small{\it 2020 Mathematics Subject 
Classification\/}: {\bf 51A05}\ \,51A20}

\begin{abstract}\noindent
In 1888, Heinrich Schroeter provided a ruler construction for points on cubic curves 
based on line involutions. Using Chasles' Theorem and the terminology
of elliptic curves, we give a simple proof of Schroeter's construction. In addition,
we show how to construct tangents and additional points on the curve 
using another ruler construction which is also based on line involutions.
\end{abstract}

\section{Introduction}
Heinrich Schroeter gave in~\cite{Schroeter} a surprisingly simple
ruler construction to generate points on a cubic curve. Since he 
did not provide a formal proof for the construction,  
we would like to present this here.
Schroeter's construction can be interpreted as an 
iterated construction of line involutions. Thus, we first define
the notion of a line involution with cross-ratios, 
and then we show how one can construct line 
involutions with ruler only. 

For the sake of simplicity,
we introduce the following terminology: For two distinct points $P$ and
$Q$ in the plane, $PQ$ denotes the line through $P$ and $Q$,
$\overline{PQ}$ denotes the distance between $P$ and~$Q$, and for 
two distinct lines $l_1$ and $l_2$, $l_1\wedge l_2$ denotes the 
intersection point of $l_1$ and $l_2$. We tacitly assume that the plane
is the real projective plane, and therefore, $l_1\wedge l_2$ is defined 
for any distinct lines $l_1$ and $l_2$.  For the cross-ratio of 
four lines $a,b,x,y$ of a pencil we use the notation $\crat(a,b,x,y)$.

\noindent{\bf Line involution.} Given a pencil. A line involution $\Lambda$ is a
mapping which maps each line $l$ of the pencil to a so-called conjugate line $\konj{l}$ 
of the pencil, such that the following conditions are satisfied:
\begin{itemize}
\item $\Lambda$ is an involution, {\ie}, $\Lambda\kringl\Lambda$ is the identity, in particular we have
$\Lambda(\konj{l})=l$.
\item Given three different pairs of conjugate lines $a,\konj{a}$,
$b,\konj{b}$, $c,\konj{c}$, and let $l_1,l_2,l_3,l_4$ be four lines among
$a,\konj{a},b,\konj{b},c,\konj{c}$ from three different pairs of conjugate
lines, then $$\crat\left(l_1,l_2,l_3,l_4\right)=\crat\left(\konj{l}_1,
\konj{l}_2,\konj{l}_3,\konj{l}_4\right).$$
\end{itemize}

Notice that any line involution is defined by two different pairs of conjugate lines.
We shall use the following construction for line involutions (for the correctness
of the construction see Chasles~\cite[Note X, \S 34, (28), p.~317]{Chasles}):
Given two pairs $a,\konj{a}$ and $b,\konj{b}$ of conjugate lines which meet in~$P$.
Suppose, we want to find the conjugate line $\konj d$ of a line $d$ from the same pencil.
Choose a point $D\neq P$ on $d$ and two lines  through $D$ which meet
$a$ and $b$ in the points $A$ and $B$, and $\konj a$ and $\konj b$ in the points
$\konj A$ and $\konj B$, respectively (see Fig.\,\ref{fig-3}).
Let $\konj D=A\konj B\wedge\konj AB$. Then the conjugate line $\konj d$ of $d$
with respect to the line involution defined by $a,\konj{a},b,\konj{b}$ is the line $P\konj D$.

Vice-versa, let $A,\konj A$ and $B,\konj B$ be two pairs of different points and
$D=AB\wedge \konj A\konj B$, $\konj D=A\konj B\wedge \konj AB$.
Then, for an arbitrary point $P\notin\{A,\konj A,B,\konj B,D,\konj D\}$, the lines $a=PA,\konj a=P\konj A$, $b=PB,\konj b=P\konj B$,
and $d=PD, \konj d=P\konj D$ are conjugate lines.
\begin{figure}[h!]
\begin{center}
\includegraphics{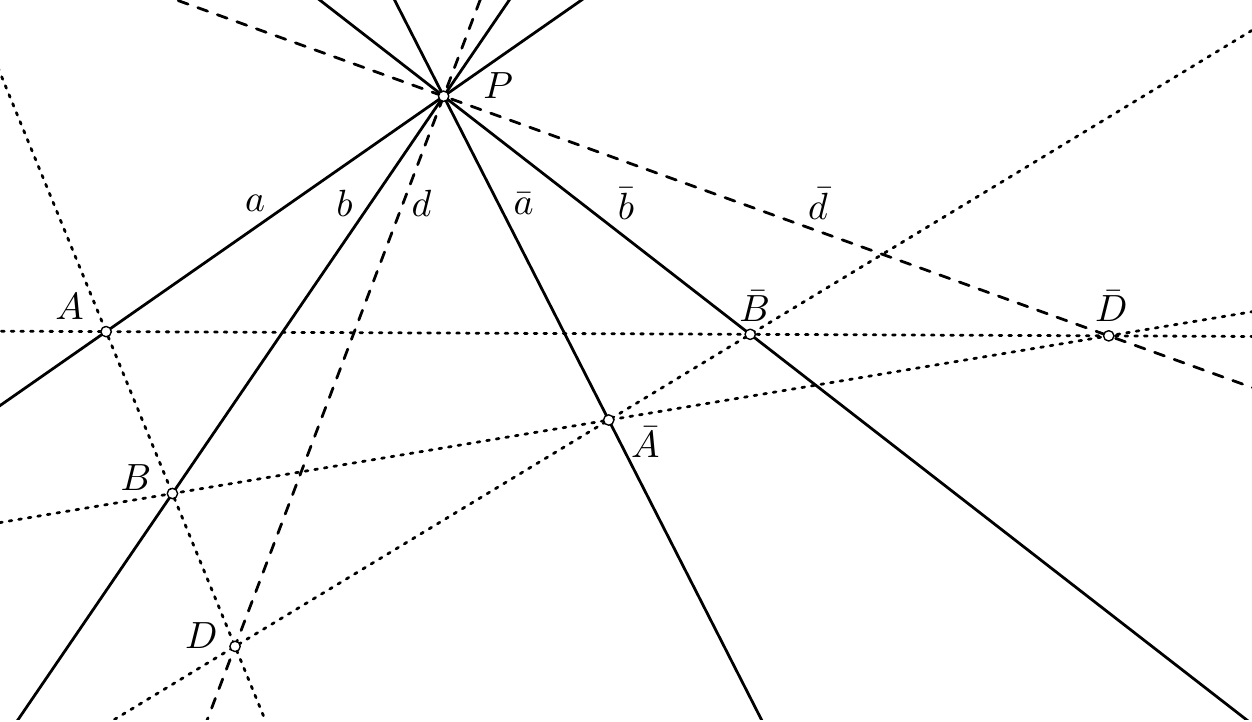}
\captionof{figure}{Construction of conjugate lines}\label{fig-3}
\end{center}
\end{figure}

Notice that this construction can be carried out using only a ruler.

\section{Schroeter's Construction for Cubic Curves}

Based on line involutions, Schroeter provided in~\cite{Schroeter} 
a simple ruler construction for cubic curves.

\begin{center}
\begin{minipage}{.9\linewidth}
{\bf Schroeter's Construction.} {\sl Let\/ 
$A,\konj{A}$, $B,\konj{B}$, $C,\konj{C}$ be six pairwise distinct points 
in a plane such that no four points are collinear
and the three pairs of points\/ $A,\konj{A}$, $B,\konj{B}$,
$C,\konj{C}$ are not the pairs of opposite vertices of the same 
complete quadrilateral. Now, for any two pairs of 
points\/ $P,\konj{P}$ and\/ $Q,\konj{Q}$,
we define a new pair\/ $S,\konj{S}$ of points by stipulating 
$$S\;:=\;PQ\wedge\konj{P}\konj{Q}\qquad\text{and}\qquad 
\konj{S}\;:=\;P\konj{Q}\wedge \konj{P}Q\,.$$
Then all the points constructed in this way lie on a cubic curve.}
\end{minipage}
\end{center}
Points $S,\konj S$ which are constructed by Schroeter's construction 
will be called Schroeter points or pairs of Schroeter points.

Notice first that with Schroeter's construction, we 
always construct pairs of conjugate lines: For any point 
$R\notin\{P,\konj P,Q,\konj Q, S,\konj S\}$ the lines
$RP, R\konj P, RQ, R\konj Q, RS, R\konj S$ are pairs of conjugate
lines with respect to the same line involution.  
Further notice that if the three pairs of 
points are opposite vertices 
of the same complete quadrilateral, then the construction gives us
no additional points.

At first glance, it is somewhat surprising that all the points we
construct lie on the same cubic curve, which is defined by three
pairs of points (recall that a cubic curve is defined by $9$~points).
The reason is that we have three {\em pairs\/} of points and not just $6$~points.
In fact, if we start with the same $6$~points but pairing them differently,
we obtain a different cubic curve. It is also not clear whether the construction
generates infinitely many points of the curve. 
Schroeter claims in~\cite{Schroeter} that this is the case, but, as we
will see in the next section, it may happen that the construction gives 
only a finite number of points.

\section{A Proof of Schroeter's Construction}

It is very likely that Schroeter discovered his construction 
based on his earlier work on cubics (see \cite{Schroeter1,Schroeter2}).
However, he did not give a rigorous proof of his construction, and the
fact that he claimed wrongly that the construction generates always infinitely many
points of the curve might indicate that he overlooked something. 
Below we give a simple proof of Schroeter's construction using
Chasles' Theorem (see Chasles~\cite[Chapitre IV, \S 8, p.\,150]{Chasles})
and the terminology of elliptic curves.

\begin{thm}[Chasles' Theorem]\label{thm:chasles}
If a hexagon\/ $ABC\konj{A}\konj{B}\konj{C}$ is inscribed in a cubic curve~$\C$
and the points\/ $AB\wedge\konj{A}\konj{B}$ and $BC\wedge\konj{B}\konj{C}$
are on~$\C$, then also\/ $C\konj{A}\wedge \konj{C}A$ is on~$\C$ (see Fig.\,\ref{fig-chasles}). 
\end{thm}

\begin{figure}[h!]
\begin{center}
\includegraphics{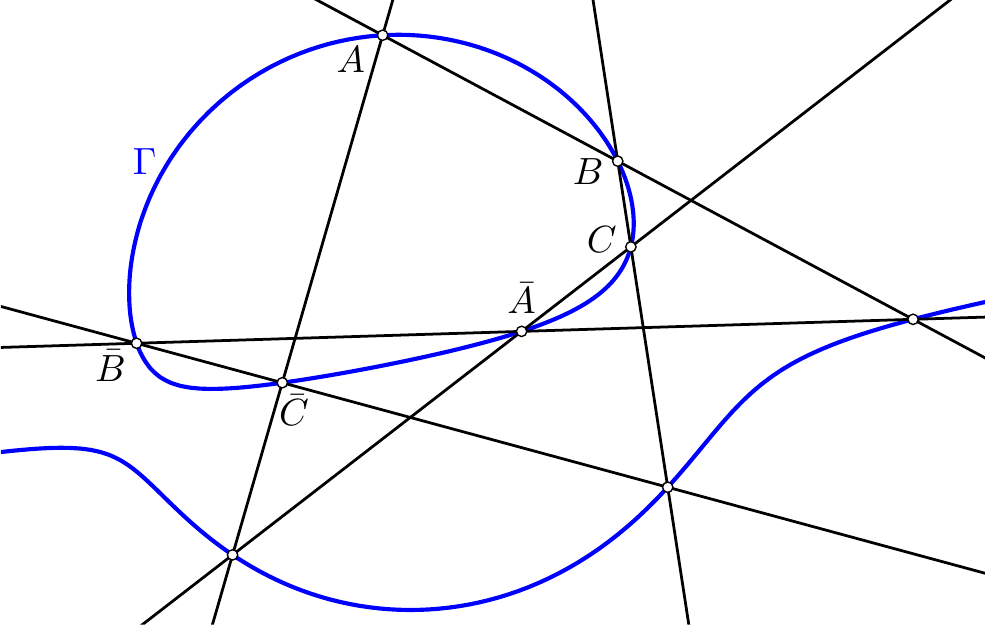}
\caption{Chasles' Theorem in a Neuberg Cubic $\C$.}\label{fig-chasles}
\end{center}
\end{figure}

With Chasles' Theorem we can prove the following 

\begin{prp}\label{prp:Chasles89}
Let\/ $A,\konj{A}$, $B,\konj{B}$, $C,\konj{C}$ be six pairwise distinct points 
in a plane such that no four points are collinear
and none of the pairs of 
points\/ $A,\konj{A}$, $B,\konj{B}$,
$C,\konj{C}$ is a pair of opposite vertices of the same complete quadrilateral.
Furthermore, let
$$
\begin{array}{rclrclrcl}
D&:=&AB\wedge\konj{A}\konj{B}\,,\qquad 
E&:=&BC\wedge\konj{B}\konj{C}\,,\qquad
F&:=&CA\wedge\konj{C}\konj{A}\,,\\[2ex]
\konj{D}&:=&A\konj{B}\wedge\konj{A}B\,,\qquad
\konj{E}&:=&B\konj{C}\wedge\konj{B}C\,,\qquad
\konj{F}&:=&C\konj{A}\wedge\konj{C}A\,,
\end{array}$$
and assume that the $9$~points\/ $A,\konj{A}$, $B,\konj{B}$, $C,\konj{C}$, $D,E,F$
are pairwise distinct and that\/ $\C$~is a cubic curve
passing through these $9$~points. Then\/ $\C$~passes also through\/ $\konj{D},
\konj{E},\konj{F}$.
\end{prp}

\begin{proof}
Since the $6$~points $A,\konj{A}$, $B,\konj{B}$, $C,\konj{C}$ as 
well as\/ $D$ and\/ $E$ are on~$\C$, by Chasles' Theorem we get that 
also\/ $\konj{F}$ is on~$\C$. Similarly,
since\/ $C,\konj{C}$, $A,\konj{A}$, $B,\konj{B}$, $F,D$ are on~$\C$,
also\/ $\konj{E}$ is on~$\C$. Finally, since 
$B,\konj{B}$, $C,\konj{C}$, $A,\konj{A}$, $E,F$ are on~$\C$,
also\/ $\konj{D}$ is on~$\C$.
\end{proof}

As an immediate consequence of Proposition\;\ref{prp:Chasles89} we get

\begin{cor}\label{cor:consequence}
The unique cubic curve\/ $\C$ passing through the $9$~points 
$A,\konj{A}$, $B,\konj{B}$, $C,\konj{C}$, $D,E,F$ contains also the 
$3$~points\/ $\konj{D},\konj{E},\konj{F}$.
\end{cor} 

In order to show that all the points constructed by Schroeter's construction 
lie on the same cubic curve, we interpret the construction in the setting
of elliptic curves. For this, let $\C$ be a cubic curve and let $\nO$ be a 
point of inflection of~$\C$\,---\,recall that every cubic curve in the real
projective plane has at least one point of inflection. For two points 
$P$ and $Q$ on $\C$ let $\dritt PQ$ be the third intersection point (counting
multiplicities) of $PQ$ with $\C$, where for $P=Q$, $PQ$ is the tangent on $\C$ 
with contact point $P$. Furthermore, for each point $P$ on $\C$, let $-P:=\dritt{\nO}P$.
As usual, we define the binary operation~$+$ on the points of~$\C$ by stipulating
$$P+Q:=-\left(\dritt PQ\right)\,.$$
Notice that $-P+P=\nO$ and, since $\nO$ is a point of inflection, we have $-\nO=\nO$.
It is well known that the operation~$+$ is associative and the structure 
$(\C,\nO,+)$ is an abelian group with neutral element~$\nO$, which is called
an {\it elliptic curve}.

Now, let $\C$ be the cubic curve passing through 
$A,\konj{A},B,\konj{B},C,\konj{C},\linebreak[1]D,\linebreak[1]E,F$ and let
$\nO$ be a point of inflection of~$\C$. Then, by construction of $\C$ we have,
for example, $\dritt AB=\dritt{\konj{A}}{\konj{B}}$, or equivalently,
$-(A+B)=-\left(\konj{A}+\konj{B}\right)$. 

\begin{lem}\label{lem:x}\begin{enumerate}[label=(\alph*)]
\item Let\/ $P,Q,\konj P,\konj Q$ be pairwise distinct points on a cubic curve~$\C$.
If\/ $S:=PQ\wedge \konj P\konj Q \in \C$ and\/ $\bar S:=P\konj Q\wedge \konj PQ\in\C$, then 
$\dritt PP=\dritt{\konj P}{\konj P}\in\C$, $\dritt QQ=\dritt{\konj Q}{\konj Q}\in\C$,
and\/ $\dritt SS=\dritt{\konj S}{\konj S}\in\C$.

\item Vice versa, if\/ $P':=\dritt PP=\dritt{\konj P}{\konj P}\in\C$ 
for two points\/ $P,\konj P\in\C$,
then we have for all\/ $Q\in\C$ the following: 
If\/ $S:=\dritt PQ$ and\/ $\konj Q=\dritt S\konj P$, then\/ 
$\konj S=P\konj Q\wedge \konj PQ\in\C$
and\/ $Q':=\dritt QQ=\dritt{\konj Q}{\konj Q}\in\C$.
\end{enumerate}
\end{lem}
\begin{figure}[h!]
\begin{center}
\includegraphics{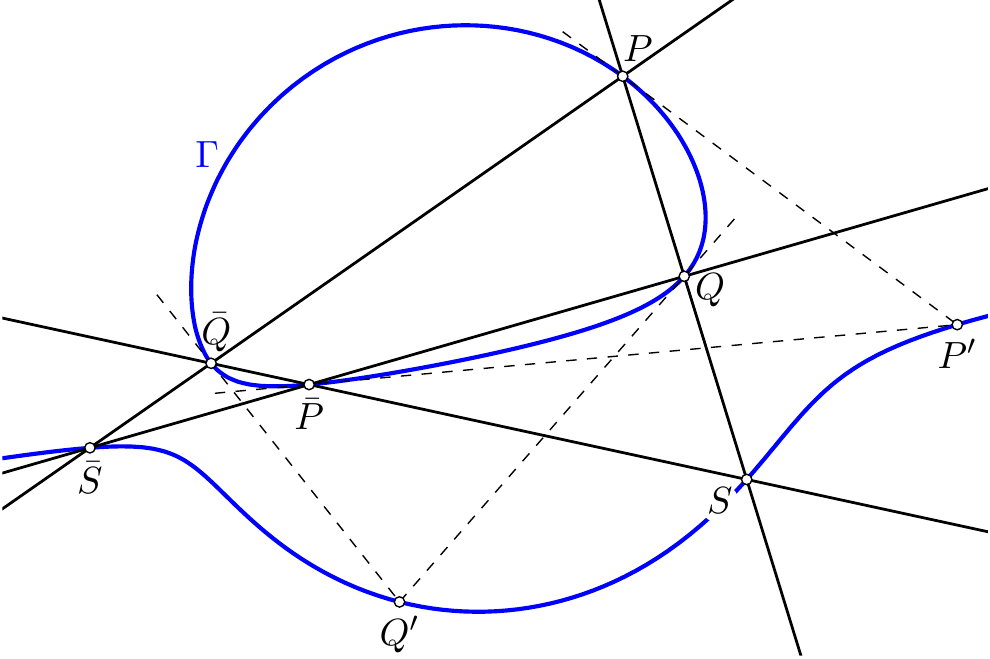}
\caption{Illustration Lemma\;\ref{lem:x}.}\label{fig-tangent}
\end{center}
\end{figure}
\begin{proof}(a)
By assumption we have $\dritt PQ=\dritt{\konj P}{\konj Q}=S$ and 
$\dritt P{\konj Q}=\dritt{\konj P}{Q}=\konj S$.
With a point $\nO\in\C$ of inflection, we get 
\begin{equation}\label{eq:1}
P+Q=\dritt{\nO}{(\dritt PQ)}=\dritt{\nO}{(\dritt {\konj P}{\konj Q})}=\konj P+\konj Q
\end{equation}
and
\begin{equation}\label{eq:2}
P+\konj Q=\dritt{\nO}{(\dritt P{\konj Q})}=\dritt{\nO}{(\dritt {\konj P}{Q})}=\konj P+Q.
\end{equation}
Adding~(\ref{eq:1}) and~(\ref{eq:2})and subtracting $Q+\bar Q$ 
yields $P+P=\konj P+\konj P$ and hence $\dritt PP=\dritt{\konj P}{\konj P}$.
Exchanging left and right hand in~(\ref{eq:1}) and adding~(\ref{eq:2}) 
gives, upon subtracting $P+\konj P$, 
$Q+Q=\konj Q+\konj Q$ and hence $\dritt QQ=\dritt{\konj Q}{\konj Q}$. 
$\dritt SS=\dritt{\konj S}{\konj S}$ follows
by exchanging the pair $Q, \konj Q$ by the pair $S,\konj S$.

(b) For the second part, we proceed as follows: By assumption, 
we have $\dritt PP=\dritt{\konj P}{\konj P}$ and therefore
$P+P=\dritt{\nO}{(\dritt PP)}=\dritt{\nO}{(\dritt{\konj P}{\konj P})}=\konj P+\konj P$. We add $S$ and subtract $P+\konj P$
to get $S+P-\konj P=S+\konj P-P$ or 
$\dritt{(\dritt{\nO}{(\dritt SP)})}{(\dritt{\nO}{\konj P})} = \dritt{(\dritt{\nO}{(\dritt S{\konj P})})}{(\dritt{\nO}P)}$.
It follows that $\dritt{(\dritt S{P})}{\konj P}=\dritt{(\dritt S{\konj P})}P$, {\ie}, 
$\dritt{Q}{\konj P}=\dritt {\konj Q}P=\konj S$. Finally, $\dritt{Q}{Q}=\dritt {\konj Q}{\konj Q}=Q'$ follows from the first part.
\end{proof}

For the sake of simplicity we write $2*P$ for $P+P$.
Let $A,\konj{A}$ be a pair of 
points 
with $\dritt AA=\dritt{\konj A}{\konj A}$
on a cubic curve $\C$,
and with respect to a given point of inflection $\nO$, 
let $T_A:=\konj{A}-A$. Then $A+T_A=\konj{A}$, which implies that
$$2*\konj{A} = 2*(A+T_A)=2*A+2*T_A.$$ 
Now, by assumption we have $2*A=2*\konj{A}$
and therefore we get that $2*T_A=\nO$. In other words, $T_A$ is
a point of order~$2$.

Now we are ready to prove the following

\begin{thm}\label{thm:main}
All the points we obtain by 
Schroeter's construction belong to the same cubic curve.
\end{thm}

\begin{proof}
Let $A,\konj{A},B,\konj{B},C,\konj{C}$ be six pairwise distinct points 
in a plane such that no four points are collinear
and none of the pairs of conjugate points $A,\konj{A}$, $B,\konj{B}$,
$C,\konj{C}$ is a pair of opposite vertices of the same complete quadrilateral.
Furthermore, let $D,\konj{D},E,\konj{E},F,\konj{F}$
be as in Proposition\;\ref{prp:Chasles89}, and let $\C$ be the cubic
curve which passes through all of these 12 points. Finally, let 
$\nO$ a fixed point of inflection of~$\C$, and let $T_A:=\konj{A}-A$,
$T_B:=\konj{B}-B$, and $T_C:=\konj{C}-C$ be three points of order~$2$. 
First we show that $T_A=T_B$. 
Since $\dritt AB=\dritt{\konj{A}}{\konj{B}}$ we have 
$-(A+B)=-(A+T_A+B+T_B)$, which implies that $T_A=T_B$. 
With a similar argument we obtain $T_B=T_C$. Thus,  we have
$T_A=T_B=T_C=:T$. 

We will say that a set $M$ of points is a {\em good set\/}, if
\begin{enumerate}[label=(\alph*)]
\item all points of $M$ belong to $\C$,\label{good0}
\item the points $A,\konj A,B,\konj B, C,\konj C,D,\konj D,E,\konj E, F,\konj F$ 
belong to $M$,\label{good1}
\item if the pair of points $S,\konj S$ belongs to $M$, then 
$S=\dritt PQ=\dritt{\konj P}{\konj Q}$ and $\konj S=\dritt P{\konj Q}=\dritt{\konj P}Q$
for two pairs $P,\konj P$ and $Q,\konj Q$ in $M$,\label{good2}
\item for all pairs $P,\konj P$ of $M$, we have $\dritt PP=\dritt{\konj P}{\konj P}$, and\label{good3}
\item for all pairs $P,\konj P$ of $M$, we have $\konj P-P=T$.\label{good4}
\end{enumerate}
Observe first, that $\{A,\konj A,B,\konj B, C,\konj C,D,\konj D,E,\konj E, F,\konj F\}$ 
is a good set.
Indeed, \ref{good0} and \ref{good1} are trivially satisfied. The property~\ref{good2} is
clear for $D,\konj D,E,\konj E, F,\konj F$. For $A$ and $\konj A$ we have
$A=\dritt BD=\dritt{\konj B}{\konj D}$, 
$\konj A=\dritt B{\konj D}=\dritt{\konj B}D$, and similarly
for the pairs $B,\konj B$ and $C,\konj C$. The property~\ref{good3} 
follows directly from Lemma\;\ref{lem:x}(a).
Finally, we have property~\ref{good4}  already for 
$A,\konj A$, $B,\konj B$ and $C,\konj C$. For 
$D$ the argument is similar: Let $T_D:=\konj D-D$.  $T_D$ is a point of order $2$
and from $B=\dritt AD=\dritt{\konj A}{\konj D}$ it follows 
$A+D=\konj A+\konj D=A+T+D+T_D$ and hence
$T_D=T$. The analogous argument shows that $\konj E-E=\konj F-F=T$.

Now suppose that $M$ is a good set, and take two pairs $P,\konj P$ and $Q,\konj Q$ in $M$.
Let $S=PQ\wedge \konj P\konj Q$ and $\konj S=P\konj Q\wedge \konj PQ$. Then we claim that
$M\cup\{S,\konj S\}$ is also a good set. We first  show that 
$\dritt PQ=\dritt{\konj P}{\konj Q}$ or equivalently that 
$P+Q=\konj P+\konj Q$. This is equivalent to
$T=\konj P-P= Q-\konj Q=T$ which is true by property~\ref{good4} for $M$ and the fact that
$T$ is a point of order 2. Then $\dritt P{\konj Q}=\dritt{\konj P}Q$ follows from Lemma\;\ref{lem:x}(b).
We conclude that the set $M\cup\{S,\konj S\}$ has the properties~\ref{good0} and~\ref{good2}.
Property~\ref{good1} is trivial. For property~\ref{good3} we need to see that $\dritt SS=\dritt{\konj S}{\konj S}$, which follows from
Lemma\;\ref{lem:x}(a). For property~\ref{good4} we define $T_S=\konj S-S$. $T_S$ is a point of order $2$.
From $Q=\dritt PS=\dritt{\konj P}{\konj S}$ it follows $P+S=\konj P+\konj S=P+T+S+T_S$ and hence
$T_S=T$. This shows that $M\cup\{S,\konj S\}$ has all properties of a good set.

It follows that all points we obtain by Schroeter's construction belong
to the same curve~$\C$.
\end{proof}

The above proof shows that the Schroeter points have the following additional properties
\begin{itemize}
\item If $P,\konj P$ is a pair of Schroeter points on $\C$, then the tangents in 
$P$ and $\konj P$ meet on~$\C$.
\item With respect to a chosen point $\nO$ of inflection, we have that $\konj P-P=T$ is 
a point of order 2 on $\C$ which is the same for all Schroeter pairs $P,\konj P$.
\end{itemize}
The following result shows that we can construct the tangent to $\C$ in each 
Schroeter point by a line involution (hence with ruler alone).
\begin{prp}\label{prp:tangent}
Let\/ $\C$ be the cubic from Proposition\;\ref{prp:Chasles89}.
Assume that\/ $S,\konj S$, $P,\konj P$, $Q,\konj Q$ are three of the pairs\/ 
$A,\konj A$, $B,\konj B$, $C,\konj C$, $D,\konj D$, $E,\konj E$, $F,\konj F$
or of the pairs which are constructed by Schroeter's construction, 
such that\/ $SP$, $SQ$, $S\konj P$, $S\konj Q$ are four distinct lines.
Let\/ $s=S\konj S$ and\/ $\konj s$ its conjugate line with respect to 
the involution given by the lines\/ $SP$, $SQ$, $S\konj P$, $S\konj Q$.
Then\/ $\konj s$ is tangent to\/ $\C$ in\/ $S$ (see Fig.\,\ref{fig:prop6}).
\end{prp}
\begin{figure}[h!]
\begin{center}
\includegraphics{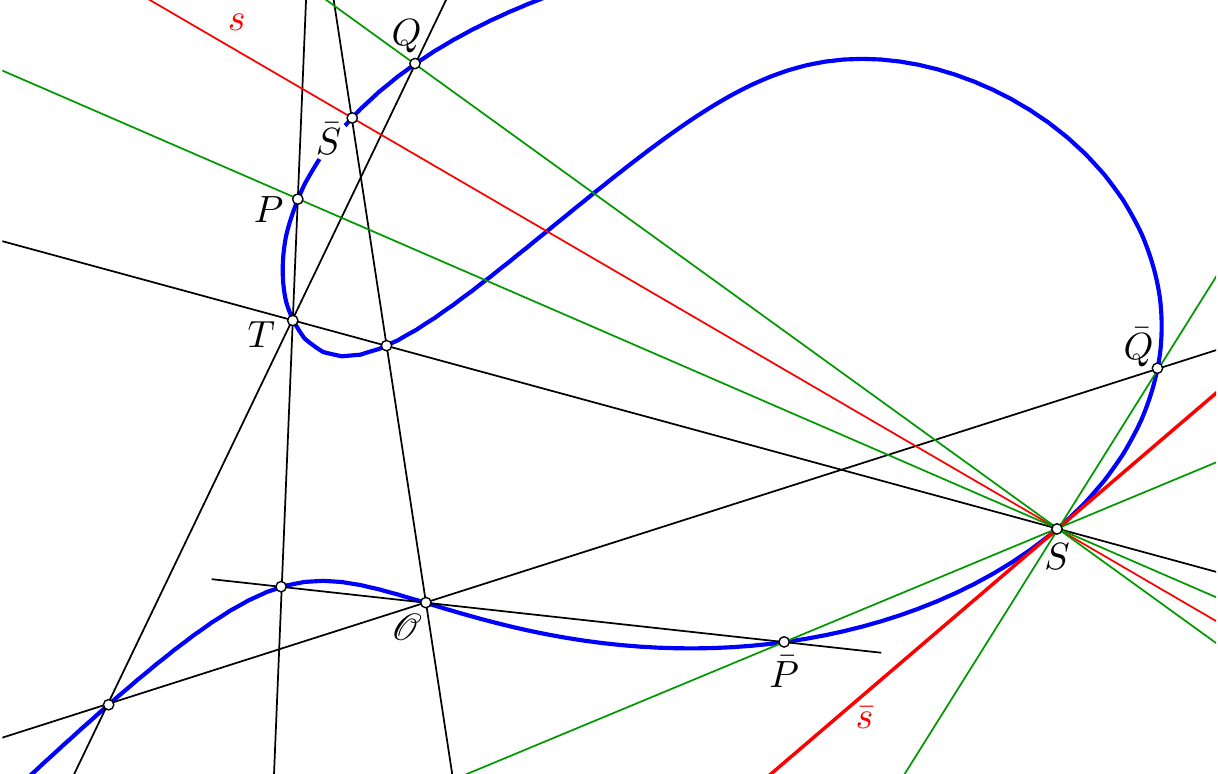}
\caption{Thin black lines: $\konj P=T+P, \konj Q=T+Q, \konj S=T+S$. The red lines
$s=S\konj S$ and the thick tangent $\konj s$ in $S$ are conjugate lines with respect to the line involution
given by the green lines $SP,S\konj P,SQ,S\konj Q$.}\label{fig:prop6}
\end{center}
\end{figure}

Before we can prove Proposition\;\ref{prp:tangent}, we have to recall 
a few facts about cubic curves.
It is well-known that every cubic curve can be transformed into 
Weierstrass Normal Form
$$
\Cab:\;
y^2=x^3+ax^2+bx$$ with $a,b,c\in\R$.
In the real projective plane, $\nO=(0,1,0)$ is a point inflection of $\Cab$
and $T_{a,b}=(0,0,1)$ is a point of order~$2$ of $\Cab$, where
$\nO$ is the neutral element of the elliptic curve~$\Cab$. 
If $A$ is a point on~$\Cab$, then we call
the point $\konj{A}:=T+A$ the {\em conjugate of\/
$A$}.
Since $T+T=\nO$, we have $$\konj{\konj{A}}=T+\konj{A}=T+T+A=\nO+A=A\,.$$
Recall that $\dritt AB:=-(A+B).$
In particular, if $C=\dritt AA$, then the line through $C$ 
and $A$ is tangent to $\Cab$ with contact point~$A$.

The following result gives a connection between conjugate points and
tangents.

\begin{fct}\label{fct:tangent}
If\/ $A,\konj{A},B$ are three points on\/ $\Cab$ which lie on a straight
line, then\/ $\dritt AA=\konj{B}$.
\end{fct}

\begin{proof} If $A,\konj{A},B$ are three points on $\Cab$ on a straight
line, then $A+\konj{A}=-B$. Thus, $A+T+A=T+A+A=-B$, which implies
$$A+A=T+(T+A+A)=T+(-B)=(-T)+(-B)=-(T+B)=-\konj{B}\,,$$ and therefore,
the line $A\konj{B}$ is tangent to $\Cab$ with contact point~$A$, {\ie},
$\dritt AA=\konj{B}$.
\end{proof}

\noindent In homogeneous coordinates, the curve $y^2=x^3+ax^2+bx$ becomes
$$\C:\;Y^2 Z=X^3+aX^2 Z+bX Z^2\,.$$ 
Assume now that $\tilde{A}=\bigl(r_0,r_1,1\bigr)$
is a point on the cubic $\C$, where $r_0,r_1\in\R\setminus\{0\}$. 
Then the point $(1,1,1)$ is on the curve
$$r_1^2Y^2 Z=r_0^3X^3+a r_0^2 X^2 Z+
b r_0 X Z^2\,.$$ Now, by exchanging $X$ and $Z$ ({\ie},
$(X,Y,Z)\mapsto (Z,Y,X)$), dehomogenizing with respect to the third 
coordinate ({\ie}, $(Z,Y,X)\mapsto (\frac ZX,\frac YX,1)$), and multiplying
with $\frac{1}{r_1^2}$, we obtain that the point $A=(1,1)$ is on the 
curve $$\Cabc:\;y^2 x=\alpha+\beta x+\gamma x^2\,,$$ where
$$\alpha=\frac{r_0^3}{{\mathstrut}^{\mathstrut}r_1^2}\,,\qquad
\beta=a\cdot\frac{r_0^2}{{\mathstrut}^{\mathstrut}r_1^2}\,,\qquad
\gamma=b\cdot\frac{r_0}{{\mathstrut}^{\mathstrut}r_1^2}\,.$$
Notice that since $A=(1,1)$ is on $\Cabc$, we have $\alpha+\beta+\gamma=1$. 

The next result gives a connection between line involutions
and conjugate points.

\begin{lem}\label{lem:involution}
Let\/ $A=(x_0,y_0)$ be an arbitrary but fixed point on\/ $\Cabc$. 
For every point\/ $P$ on\/ $\Cabc$ which is different from\/ $A$ and\/ $\konj{A}$,
let\/ $g:=AP$ and\/ $\konj{g}:=A\konj{P}$. Then the mapping\/ $\Inv{A}:g\mapsto\konj{g}$
is a line involution.
\end{lem}

\begin{proof} It is enough to show that there exists a point $\zeta_0$ 
(called the center of the involution) on
the line $h:x=0$, such that the 
product of the distances between $\zeta_0$ and the intersections of $g$
and $\konj{g}$ with~$h$ is constant. 
\begin{figure}[h!]
\begin{center}
\includegraphics{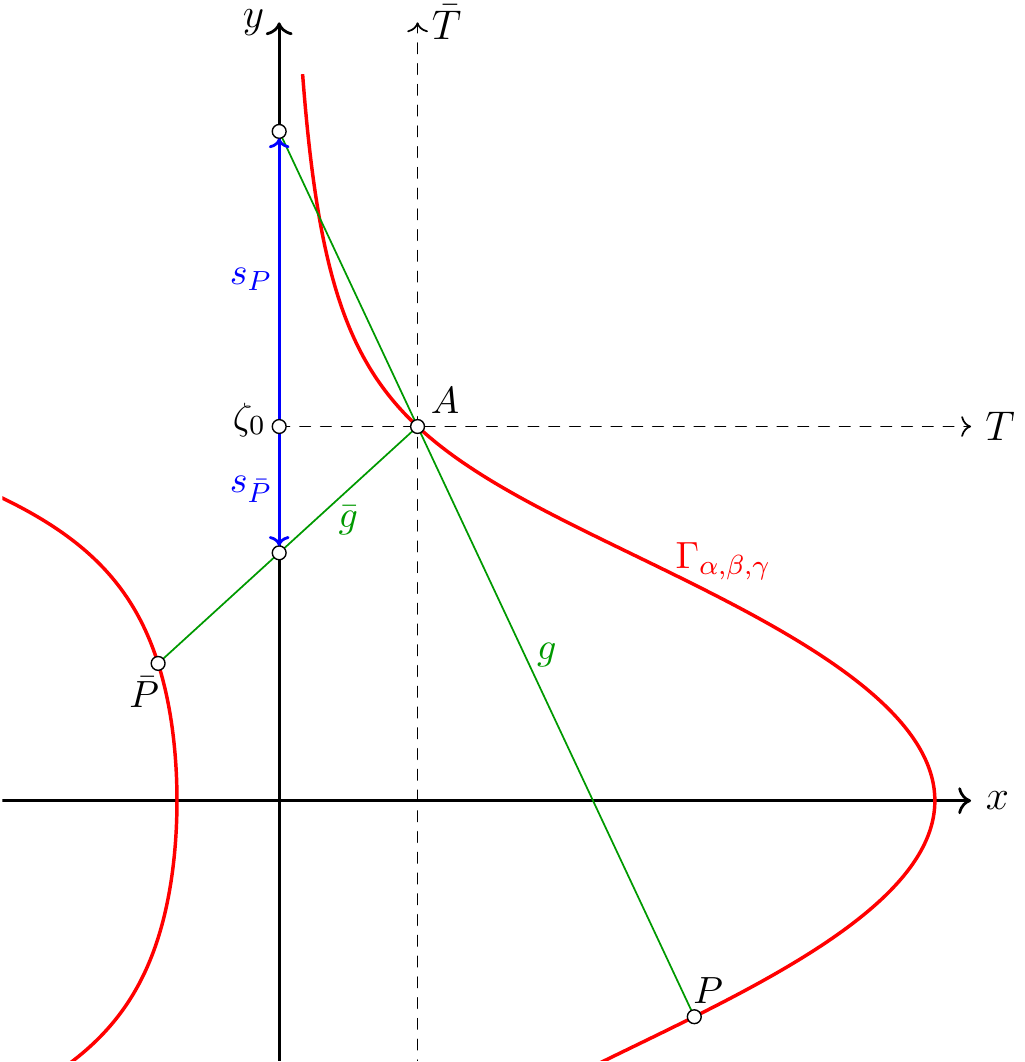}
\caption{Line involution.}\label{fig:involution}
\end{center}
\end{figure}

Since $\konj{T}=T+T=\nO$, with respect to
$T$ we have $g: y=y_0$ and $\konj{g}:x=x_0$, which implies that 
$\zeta_0=(0,y_0)$. Now, let $P=(x_1,y_1)$ be a point on $\Cabc$ which is 
different from $A,\konj{A},T,\nO$, and let $g:=AP$ and $\konj{g}:=A\konj{P}$.
Since $\konj{P}=\bigl(\frac{\alpha}{\gamma x_1},-y_1\bigr)$, the slopes 
$\lambda_P$ and $\lambda_{\konj{P}}$ of $g$ and $\konj{g}$, respectively, are
$$\lambda_P=\frac{y_1-y_0}{x_1-x_0}
\qquad\text{and}\qquad
\lambda_{\konj{P}}=\frac{-y_1-y_0}{\frac{\alpha}{\gamma x_1}-x_0}\,.$$
Thus, the distances $s_P$ and $s_{\konj{P}}$ between $\zeta_0$ 
and the intersections of $g$ and $\konj{g}$ with~$h$, respectively, are
$$s_P=-\frac{x_0(y_1-y_0)}{x_1-x_0}
\qquad\text{and}\qquad
s_{\konj{P}}=\frac{x_0(y_1+y_0)\cdot\gamma x_1}{\alpha-\gamma x_1 x_0}\,.$$
Now, $$s_P\cdot s_{\konj{P}}=
-\frac{x_0^2(y_1-y_0)(y_1+y_0)\gamma x_1}{(x_1-x_0)(\alpha-\gamma x_0 x_1)}=
-\frac{x_0^2(y_1^2-y_0^2)\gamma x_1}{(x_1-x_0)(\alpha-\gamma x_0 x_1)}\,,$$
and using the fact that for $i\in\{0,1\}$, $y_i^2=\frac{\alpha}{x_i}+\beta+\gamma x_i$,
we obtain $$s_P\cdot s_{\konj{P}}=\gamma\cdot x_0\,,$$
which is independent of the particular point $P=(x_1,y_1)$.
\end{proof}

Since line involutions are invariant under projective transformations,
as a consequence of Lemma\;\ref{lem:involution} we obtain the following 

\begin{fct} 
Let\/ $\C$ be the cubic from Proposition\;\ref{prp:Chasles89}
with two pairs of Schroeter points\/ $P,\konj P=T+P$, $Q,\konj Q=T+Q$,
and let\/ $R$ be a point on\/ $\C$ such that\/ $RP,R\konj P,RQ,R\konj Q$ are
four different lines. Let\/ $S$ be a further point on\/ $\C$ and\/ $\konj S=T+S$.
Then the lines\/ $s=RS$ and\/ $\konj s=R\konj S$ are conjugate lines with respect
to the line involution given by the lines\/ $RP,R\konj P,RQ,R\konj Q$ 
(see Fig.\,\ref{fig:fact9}).
\end{fct}

\begin{figure}[h!]
\begin{center}
\includegraphics{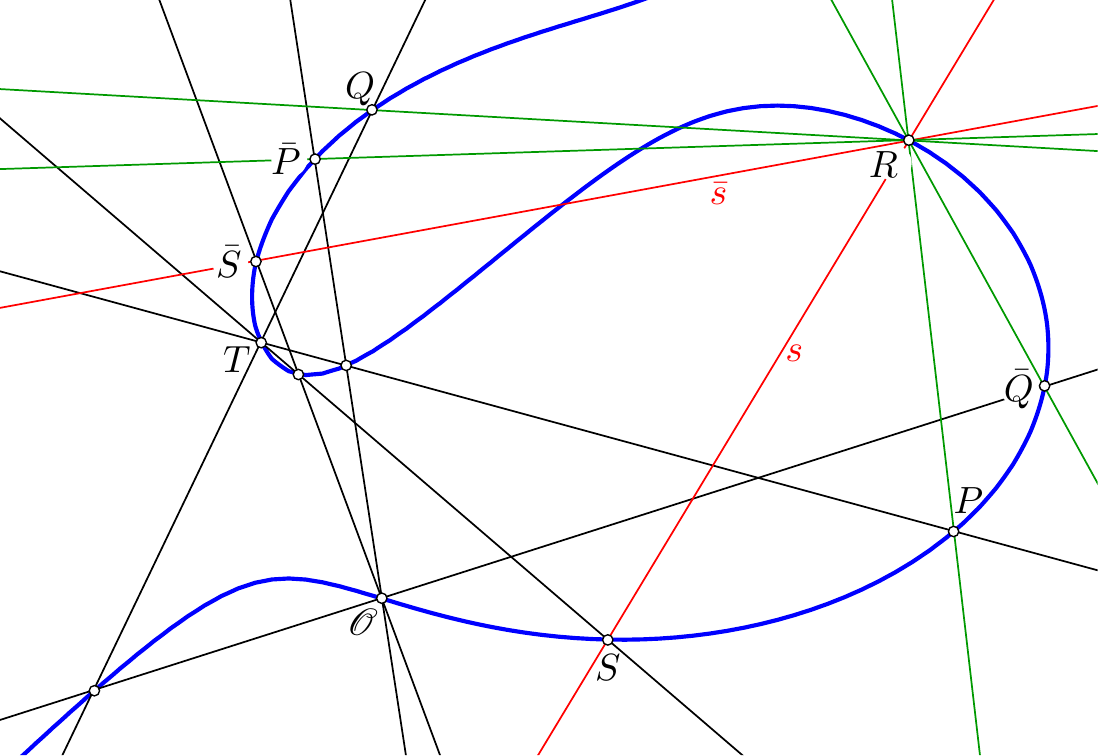}
\caption{Thin black lines: $\konj P=T+P$, $\konj Q=T+Q$, $\konj S=T+S$.
The red lines $s=RS$ and $\konj s=R\konj S$ are conjugate lines with respect
to the line involution given by the green lines $RP,R\konj P,RQ,R\konj Q$.}\label{fig:fact9}
\end{center}
\end{figure}

Now we are ready to prove Proposition\;\ref{prp:tangent}.

\begin{proof}[Proof of Proposition\;\ref{prp:tangent}] 
First notice that $S$ and $\konj{S}$ are distinct,
since otherwise, $\konj{S}=T+S=S$, which implies that $T=S-S=\nO$.

Assume that the line $s$ intersects $\C$ in a point $U$ which
is different from $S$ and $\konj S$. Then $\konj U:=T+U$ belongs to $\konj s$.
If the line $\konj s$ intersects $\C$ in a point $V$ which
is different from $\konj U$, then, with respect to the 
involution given by the lines $\konj U S$, $\konj U P$, 
$\konj U\konj S$, $\konj U\konj P$, 
the point $\konj V$ belongs to~$s$. Hence, 
$\konj V=\konj S$, which shows that $\konj s$ 
is tangent to $\C$ in $S$. 

Now, assume that the line $s$ intersects $\C$ just in $S$ and $\konj S$.
Then, the line $s$ is tangent to $\C$ either in $S$ or in $\konj S$. 
We just consider the former case, the latter case is handled similarly.
Let $P_n$ (for $n\in\N$) be a sequence of points on $\C$ which are different
from $S$ and which converges to~$S$, {\ie}, $\lim_{n\to\infty}P_n=S$.
Since for each $n\in\N$ we have $\konj{P}_n=T+P_n$ (where $\konj{P}_n:=T+P$), 
by continuity of addition 
we have $\lim_{n\to\infty}\konj{P}_n=\konj S$. 
For each $n\in N$ let $t_n:=P_nS$. 
Then, for each $n\in\N$, $\konj{t}_n=\konj{P}_nS$. 
Since $s$ is tangent to $\C$ in $S$,
by continuity, on the one hand we have $\lim_{n\to\infty}t_n=s$, and
on the other hand we have $\lim_{n\to\infty}\konj{t}_n=s$, 
which implies that $\konj s=s$ and shows that $\konj s$ 
is tangent to $\C$ in~$S$. 
\end{proof}

As a corollary of Proposition\;\ref{prp:tangent} 
and
Lemma\;\ref{lem:x}(a) we obtain 
the following:
\begin{cor} Let $\C$ be the cubic from Proposition~\ref{prp:Chasles89}. Then we have:
\begin{enumerate}[label=(\alph*)]
\item In each Schroeter point it is possible to construct the tangent by a line involution, {\ie}, with a ruler construction.
\item In addition to the Schroeter points on $\C$ one can construct for each Schroeter pair 
$P,\konj P$ the point $\dritt PP=\dritt{\konj P}{\konj P}\in\C$ by ruler alone:
These are  the intersection points of the tangents in $P$ and in $\konj P$.
\end{enumerate}
\end{cor}
Fig.\,\ref{fig:points} shows a sample of Schroeter points and of intersection points of the corresponding tangents.
\begin{figure}[h!]
\begin{center}
\includegraphics[width=.6\textwidth]{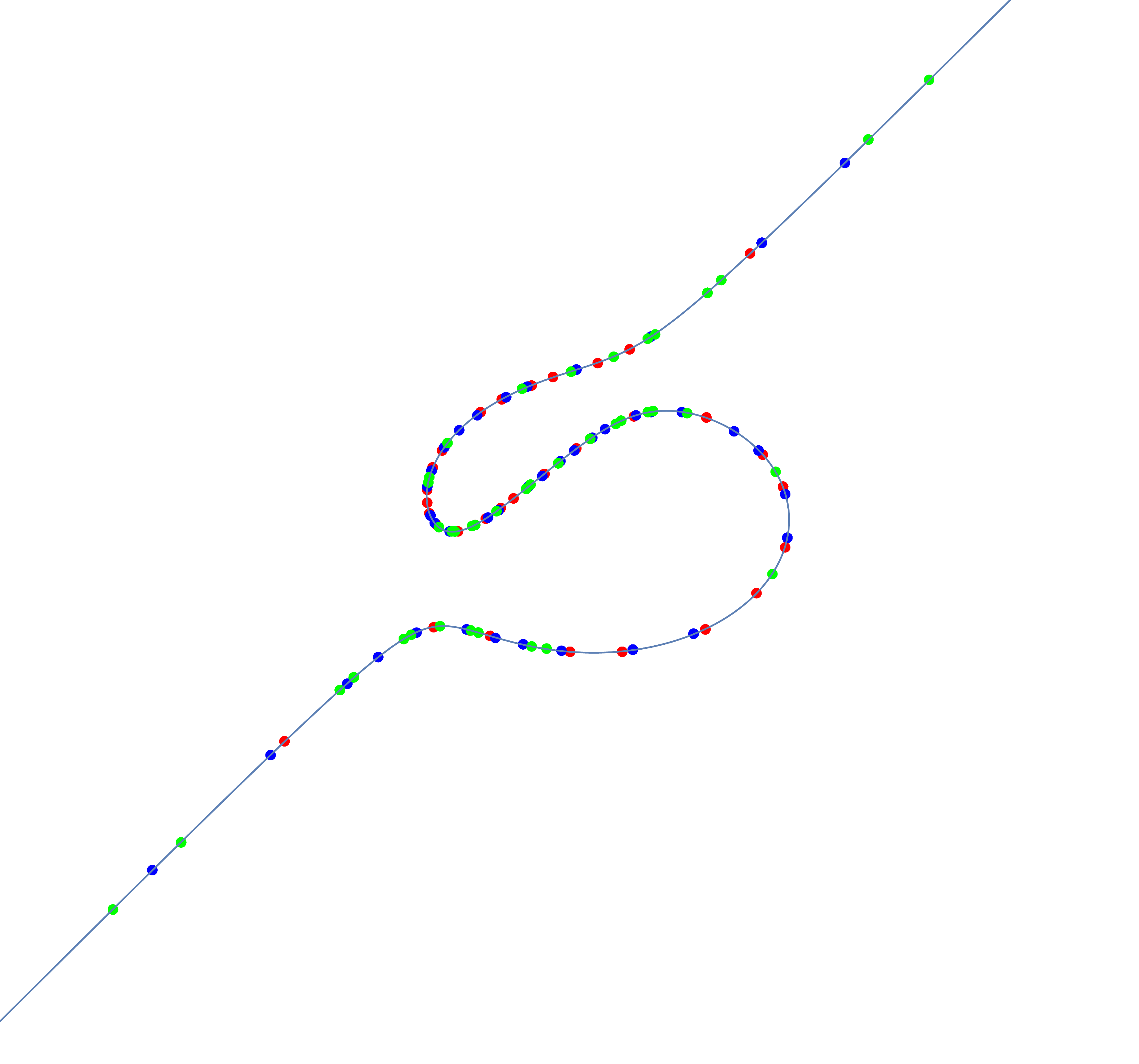}
\caption{Schroeter pairs $P$ (red), $\konj P$ (blue), and  intersection points $\dritt PP=\dritt{\konj P}{\konj P}$ of the corresponding tangents (green).}\label{fig:points}
\end{center}
\end{figure}

A priori it might be possible that Schroeter's construction does not yield {\em all\/}
cubic curves. However, the next theorem says that in fact all cubic
curves carry Schroeter's construction.

\begin{thm}
Let\/ $\C$ be a non-singular cubic curve. Let\/ $A,B,C$ be three different arbitrary points 
on~$\C$. Then, there are points\/ $\konj A,\konj B,\konj C$ on\/ $\C$ such that\/
$D=AB\wedge \bar A\bar B$, $E=BC\wedge \bar B\bar C$, $F=CA\wedge \bar C\bar A$
are points on\/ $\C$ and so do all the points given by Schroeter's construction.
\end{thm}

\begin{proof}
Choose $\bar A$ such that  $\dritt AA = \dritt{\bar A}{\bar A}$ and
$\konj B:=\dritt{\konj A}({\dritt AB})$. In particular, we have $\dritt AB=\dritt{\konj A}{\konj B}$,
and, by Lemma\;\ref{lem:x}, $\dritt A{\konj B}=\dritt{\konj A}B$ and $\dritt BB=\dritt{\konj B}{\konj B}$.
Let $\konj C:=\dritt{\konj B}({\dritt BC})$. In particular, we have $\dritt BC=\dritt{\konj B}{\konj C}$,
and, by Lemma\;\ref{lem:x}, $\dritt B{\konj C}=\dritt{\konj B}C$ and $\dritt CC=\dritt{\konj C}{\konj C}$.
It follows from Chasles' Theorem\;\ref{thm:chasles} that $\dritt A{\konj C} = \dritt{\konj A}C$.
From the above, we obtain by applying Proposition\;\ref{prp:Chasles89} with
$C$ and $\konj C$ exchanged, that $\dritt AC = \dritt{\konj A}{\konj C}$.
Hence all points constructed from these points by Schroeter's construction lie on $\C$.
\end{proof}
\noindent{\bf Remarks.} Let $\C_0$ be the cubic curve passing through 
$A,\konj{A}$, $B,\konj{B}$, $C,\konj{C}$, $D,E,F$, let $\nO$ be a point of
inflection of $\C_0$, and let $E_0=(\C_0,\nO,+)$ be the corresponding elliptic
curve.
\begin{enumerate}
\item[(1)] If $C_n$ is a cyclic group of order $n$, then there is a point
on $\C_0$ of order~$n$ (with respect to $E_0$). This implies that if we
choose the six starting points in a finite subgroup of $E_0$, then 
Schroeter's construction ``closes'' after finitely many steps and we end up with
just finitely many points. However, if our~$6$ starting points are all rational
and we obtain more than~$16$ points with Schroeter's construction, then,
by Mazur's Theorem, we obtain infinitely many rational points on the cubic
curve~$\C_0$.

\item[(2)] If the elliptic curve $E_0$ has three points of order~$2$, then one
of them, say~$T$, has the property that for any point $P$ on $\C_0$ we have
$\konj{P}=P+T$. In particular, we have $\konj{T}=T+T=\nO$.
Furthermore, for the other two points of order~$2$, say $S_1$ and $S_2$, we have
$S_1=S_2+T$ and $S_2=S_1+T$, {\ie}, $S_1=\konj{S_2}$. 

\item[(3)] If we choose another point of inflection $\nO'$ on
the cubic curve $\C_0$, we obtain a different elliptic curve~$E_0'$.
In particular, we obtain different inverses of the constructed points, 
even though the constructed points are exactly the same (see Fig.\,\ref{fig-4}).

\begin{figure}[h!]
\begin{center}
\includegraphics{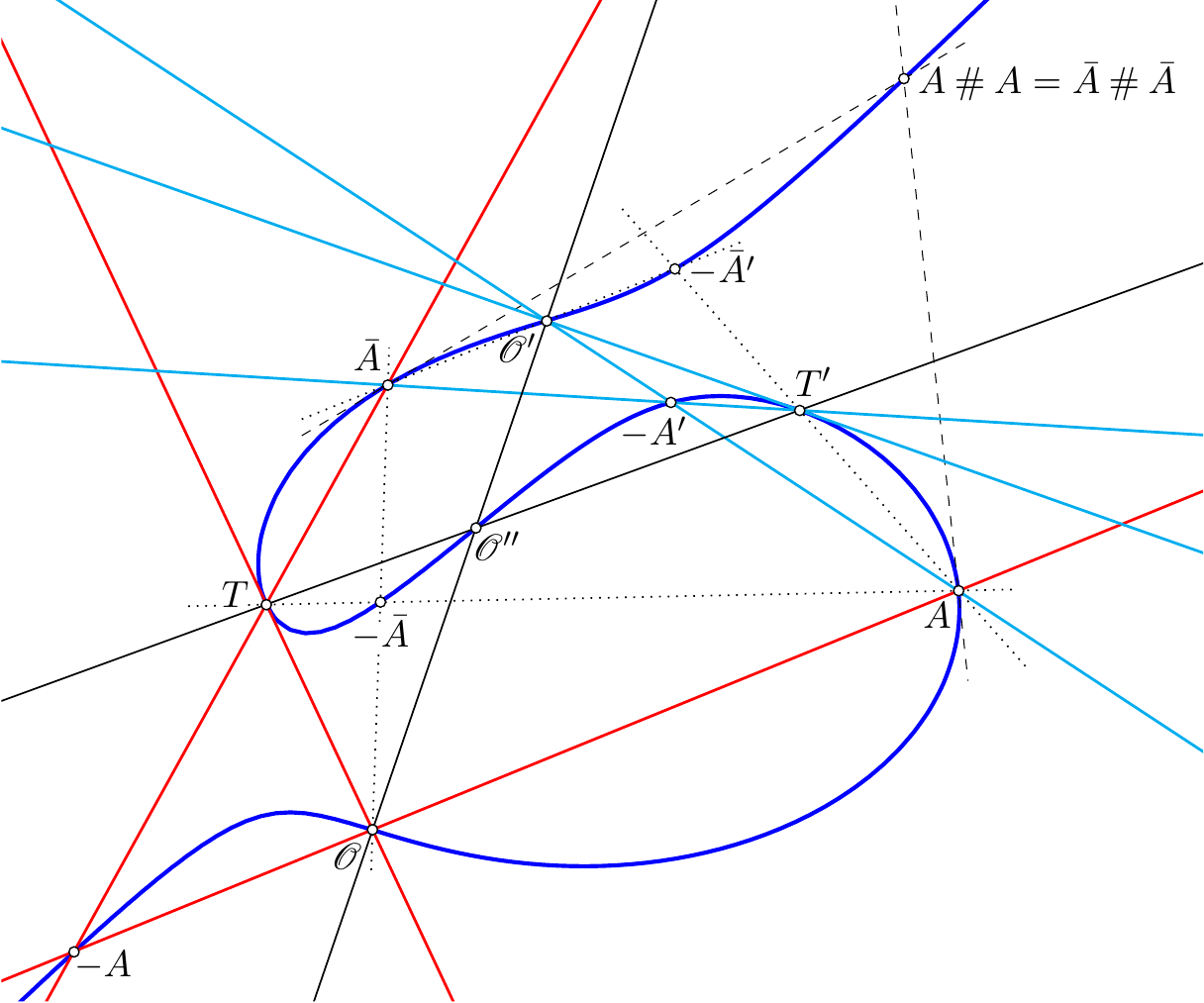}
\caption{$A,\konj A$ is a Schroeter pair, 
in particular $\dritt AA=\dritt{\konj A}{\konj A}$.
With respect to $\nO$, we get the points $-A, -\konj A$, 
and the point $T=\konj A-A = A-\konj A$ of order 2.
With respect to $\nO'$, we get the points $-A', -\konj{A'}$, 
and the point $T'=\konj A-A'= A-\konj{A'}$ of order 2.
The three points $\nO,\nO',\nO''$ of inflection are collinear.
The lines $TT'$ and $\nO\nO'$ meet in $\nO''$.
}\label{fig-4}
\end{center}
\end{figure}
\end{enumerate}
\noindent{\bf Example.} Let $A,\konj{A},B,\konj{B},C,\konj{C}$ 
be six different starting points for Schroeter's construction
such that no three points are co-linear.
By a projective transformation, we can move 
$A\mapsto (0, 0, 1)$,
$\konj{A}\mapsto (0, 1, 0)$,
$B\mapsto (1, 0, 0)$,
$\konj{B}\mapsto (1, 1, 1)$,
$C\mapsto (C_x,C_y,1)$,
$\konj{C}\mapsto (\konj{C}_x, \konj{C}_y, 1)$. Then, the corresponding
cubic curve $\C$ we obtain by Schroeter's construction is given by
the following equation:

\begin{multline*}\C:\quad 
x y^2 - x^2 y
+ x^2\,C_y \konj{C}_y 
+ y^2\left(C_x \konj{C}_x-C_x - \konj{C}_x \right)\\[1.7ex]
+ x y \left(C_x + \konj{C}_x - C_y \konj{C}_x - C_x \konj{C}_y\right)
- x\,C_y \konj{C}_y
+ y \left(C_y \konj{C}_x + C_x \konj{C}_y-C_x \konj{C}_x\right)=0
\end{multline*}


\bibliographystyle{plain}

\end{document}